\documentclass{amsart}
\usepackage[T1]{fontenc}
\usepackage[a4paper,lmargin={2.3cm},rmargin={2.3cm},tmargin={2.3cm},bmargin = {2.6cm}]{geometry}
\usepackage{xcolor}
\usepackage{amssymb}
\usepackage{amsthm}
\usepackage{graphicx}
\usepackage{enumerate}
\usepackage{amsmath}
\usepackage{dsfont}
\usepackage[mathscr]{euscript}
\usepackage{mathrsfs}
\usepackage{aligned-overset}
\usepackage{bbm}
\usepackage{shuffle}
\usepackage[colorlinks=true, 
linkcolor=OceanBlue, 
filecolor=magenta,       
urlcolor=MildBlue, 
citecolor=orange, 
]{hyperref} 
\definecolor{OceanBlue}{HTML}{014f86}  
\definecolor{rost}{HTML}{9d0208}    
\definecolor{MildBlue}{HTML}{457b9d}  
\definecolor{ForestGreen}{HTML}{386641}
\definecolor{orange}{HTML}{fb5607} 

\usepackage[backend=biber, style=alphabetic, natbib=true, hyperref=true, giveninits=true, maxbibnames=99]{biblatex}
\usepackage{csquotes}

\theoremstyle{plain}
\newtheorem{lem}{Lemma}[section]
\newtheorem{theo}[lem]{Theorem}
\newtheorem{cor}[lem]{Corollary}
\newtheorem{defi}[lem]{Definition}

\theoremstyle{definition}
\newtheorem{rem}[lem]{Remark}
\newtheorem{exmp}[lem]{Example}

\usepackage{xargs}  
\usepackage[colorinlistoftodos,prependcaption,textsize=tiny]{todonotes}
\newcommandx{\add}[2][1=]{\todo[disable,inline, linecolor=ForestGreen,backgroundcolor=ForestGreen!25,bordercolor=ForestGreen,#1]{#2}}
\newcommandx{\unsure}[2][1=]{\todo[inline, linecolor=rost,backgroundcolor=rost!25,bordercolor=rost,#1]{#2}}

\newcommand{\one}{\mathbf{1}}

\newcommand{\dd}{\mathop{}\!\mathrm{d}}

\newcommand{\partition}{\mathcal{P}}

\newcommand{\id}{\mathrm{id}}

\newcommand{\Xb}{\mathrm{\mathbf{X}}}

\newcommand{\lb}{\left(}
\newcommand{\rb}{\right)}

\newcommand{\vertiii}[1]{{\left\vert\kern-0.25ex\left\vert\kern-0.25ex\left\vert #1 \right\vert\kern-0.25ex\right\vert\kern-0.25ex\right\vert}}

\newcommand{\qv}[2]{\left[#1\right]_{#2}}
\newcommand{\sig}[2]{\mathrm{Sig}{(#1)}_{#2}}
\newcommand{\proj}{\mathcal{P}}
\newcommand{\sym}[1]{\mathrm{Sym}\left(#1\right)}
\newcommand{\A}{\mathcal A}
\newcommand{\la}{\langle}
\newcommand{\ra}{\rangle}

\DeclareMathOperator{\borel}{\EuScript{B}}

\newcommand{\X}{\mathbb{X}}

\DeclareMathOperator{\R}{\mathbb{R}}

\DeclareMathOperator{\N}{\mathbb{N}}

\title{Rough Functional It\^o Formula}
\author{Franziska Bielert}
\address{Institut f\"ur Mathematik, Technische Universit\"at, Berlin}
\date{}

\begin{document}
\begin{abstract}
We prove a rough It\^o formula for path-dependent functionals of $\alpha$-H\"older continuous paths for $\alpha\in(0,1)$. Our approach combines the sewing lemma and a Taylor approximation in terms of path-dependent derivatives.
\end{abstract}
\maketitle

\section{Introduction}
In \parencite{dupireFunctionalItoCalculus2009} Dupire developed an It\^o calculus for causal functionals $F$, i.e., functionals that depend at time $t\in[0,T]$ on a path $X\colon [0,T] \to \R$ up to time $t$. He introduced suitable notions of directional derivatives. A `time' derivative $DF$ by a perturbation of time $t+h$ with stopped paths $X_t$ and a `space' derivative $\nabla F$ by fixing the time and perturbing of the end point of the stopped path $X_t + h\one_{[t,T]}$. Similar results were established in a number of papers by R. Cont and D.A. Fourni\'e using purely analytical arguments for paths $X$ that have finite quadratic variation in a pathwise sense. A thorough treatment can be found in \parencite{BallyVlad2016Sibp}. In particular they proved a pathwise functional It\^o formula using F\"ollmer type integrals introduced in \parencite{Föllmer1981}.

Afterwards the pathwise It\^o formula for path-dependent functionals (as well as for standard functions) was extended to paths $X$ with arbitrary regularity by R. Cont and N. Perkowski in \parencite{contPathwiseIntegrationChange2019}. For non path-dependent functions they also investigated the relation to rough path theory. By identifying a natural candidate for the reduced rough path $\Xb$ induced by a multidimensional path $X$, it was shown that the F\"ollmer integral in the pathwise It\^o formula coincides with a rough integral. 

The main result of this paper is Theorem \ref{theo:roughfuncInt}. It constructs a rough integral $\int \nabla F(t, X) \dd\Xb(t)$ for multidimensional $\alpha$-H\"older continuous paths for $\alpha\in(0,1)$ and provides a rough functional It\^o formula under suitable regularity assumptions on the causal functional $F$.  

Similar results for less regular functionals $F$ and cadlag paths $X$ with finite $p$-variation have been obtained independently by Christa Cuchiero, Xin Guo and Francesca Primavera and are to be published in \parencite{cuchieroFunctionalItoformulaTaylor}. Their proof relies on a density argument, passing from linear functions of the
signature of the path to general path functionals.

This work was intended to give a simple proof that follows the standard approach in rough path theory. So we will allow for strong regularity assumptions on $F$ and apply the sewing lemma. 
Namely Corollary \ref{cor:remainderestimate} gives an error bound for higher order Taylor approximations of $F(t, X)$ in terms of the causal derivatives. This is a generalization of Lemma 2.2 from A. Ananova and R. Cont in \parencite{ananovaPathwiseIntegrationRespect2017}. The higher order Taylor approximation allows to adapt the techniques in \parencite{contPathwiseIntegrationChange2019} to the path-dependent setting.

\subsection{Notation}
Let $T>0$ and $D$ denote the set of cadlag paths $X\colon [0, T] \to \R^d$ equipped with the uniform norm $\vert \cdot \vert_\infty$. For such paths and $t\in [0,T]$ we denote by $X(t)$ the value of the path at time $t$ and by $X_t$ the stopped path $X_t = X(\cdot \wedge t)$. Let further $X_{t-}$ denote the path $X$ stopped right before $t$, namely for $u\in[0,T]$, $X_{t-}(u) = X(u)\one_{[0,t)}(u) + \lim_{r\uparrow t}X(r) \one_{[t, T]}(u)$ .

Let $\Delta_T:= \{ (s,t)\in [0,T]\times [0,T]\colon 0\le s\le t\le T\}$.
We call $\partition = \{ [t_{k-1}, t_k]\colon k=1,\dots,n\}$ with $t_k\in[0,T]$ for all $k=0,\dots,n$, \textit{partition of $[0,T]$} if $0=t_0< t_1<\dots < t_n = T$. The \textit{mesh} of a partition $\partition$ is defined as $\vert\partition\vert = \max_{[s,t]\in\partition} \vert t-s\vert$. \index{interval@$\vert\partition\vert$}

For $\alpha\in(0,1)$, a two-parameter path $\Xi\colon\Delta_T\to \R^d$ is $\alpha$-H\"older continuous if 
\begin{align*}
\vert \Xi\vert_\alpha := \sup_{\substack{(s,t)\in\Delta_T,\\ s<t}} \frac{\vert\Xi(s,t)\vert}{\vert t-s\vert^\alpha} < \infty,
\end{align*}
here $\vert \cdot\vert$ denotes the euclidean norm.
Then a path $X\colon [0,T]\to \R^d$ is $\alpha$-H\"older continuous if its increments $(\delta X)(s,t):= X(t) - X(s)$ are.

For two terms $x,y$ we abbreviate the existence of some constant $C>0$ such that $x\le C y$ to $x\lesssim y$ and by $\lesssim_p$ we indicate a dependency $C= C(p)$ on some parameter $p$.\index{$\lesssim$}

\subsection{Causal Derivatives}
Following \parencite{dupireFunctionalItoCalculus2009} and \parencite{oberhauserExtensionFunctionalIto2012} (from where we took the present definitions), we consider for \textit{causal} functionals $F\colon\ [0,T]\times D \to \R$, i.e., 
$F(t, X) = F(t, X_t)$,  the following notions of differentiability:

\begin{defi}[Causal Space Derivative]\label{defi:D}
	If for all $(t,X)\in [0,T]\times D$ the map
	\begin{align*}
		\R^d\ni h\mapsto F(t, X_t + h \one_{[t,T]})
	\end{align*}
	is continuously differentiable at $h = 0$ we say that $F$ has a causal space derivative. We denote it by $\nabla F(t,X) = (\partial_1 F(t,X), \dots, \partial_n F(t,X))$. Similarly, we define for $n\in\N$ the $n$th causal space derivative and denote it by $\nabla^n F$.
\end{defi} 
\begin{defi}[Causal Time Derivative]\label{defi:nabla}
	If for all $(t,X)\in [0,T]\times D$ the map 
	\begin{align*}
		[0,\infty)\ni h \mapsto F(t + h, X_t)
	\end{align*}
is continuous and right-differentiable at $h=0$ we denote this derivative by $D F(t, X)$. If additionally $t\mapsto D F(t, X)$ is Riemann integrable, then we say that $F$ has a causal time derivative.
\end{defi}

For $n\in\N$ we write $F\in\mathbb{C}^{1,n}_b$, if $F$ has a causal time derivative and $n$ causal space derivatives such that $F$, $DF$ and for $k=1, \dots, n$, $\nabla^k F$ are continuous in $[0,T]\times D$ and bounded in the sense that $\sup_{(t, X)\in[0,T]\times D} \vert F(t, X)\vert <  \infty$. 
We refer to \parencite[Definition 19]{oberhauserExtensionFunctionalIto2012} for weaker regularity notions. Since the purpose of this paper is to give a simple proof of a rough functional It\^o formula, we keep the assumptions simple. 

\section{Taylor Approximation for Causal Functionals}

To derive a Taylor formula for $t \mapsto F(t, X)$ we use the signature $\sig{X}{}$ of a paths that have bounded variation. 
We briefly specify the (for us necessary) theory.

\subsection{Symmetric Part of the Signature of a Path}
Set $T_0(\R^d) := 1$ and for $k\in\N$ , $T_k(\R^d) := (\R^d)^{\otimes k}$ the space of $k$-tensors and $T(\R^d) = \bigoplus_{k=0}^\infty T_k(\R^d)$ the \textit{tensor algebra}. 
A \textit{word} $w$ in the alphabet $\A := \lbrace 1, \dots d\rbrace$ of length $k$ is a tuple $(w_1, \dots, w_k)$ such that for $j=1, \dots, k$, $w_j \in \A$. 
Denote for $i=1, \dots, d$ by $e_i := (0, \dots, 0, 1, 0, \dots, 0)$ the $i$th unit vector and $e_w := e_{w_1}\otimes \dots, \otimes e_{w_k}$. Then the set $\lbrace e_w\colon w \text{ word in } \A \text{ of length }k\rbrace$ is a basis of $T_k(\R^d)$. 
We write $\langle\cdot, \cdot\rangle$ for the natural inner product in $T_k(\R^d)$. 
Abusing the notation a bit, we also write for $T\in T_k(\R^d)$, $S\in T_m(\R^d)$ with $m<k$, $\la T, S\ra \in T_{m-k}$, where for $h\in T_{k-m}(\R^d)$, $\la T, S\ra (h) := \la T, S \otimes h\ra$. Finally note that we can choose \textit{compatible} norms $\vert \cdot\vert$ on $T_k(\R^d)$, i.e. for $v_1, \dots, v_k\in\R^d$, 
\begin{align*}
\vert v_1 \otimes \dots \otimes v_k\vert \le \prod_{j=1}^k \vert v_j\vert.
\end{align*}
Let $\proj_k$ denote the projection from $T(\R^d)$ onto $T_k(\R^d)$. This extends naturally to infinite series $T((\R^d)):= \prod_{k=0}^\infty T_k(\R^d)$.\\

Let further $X\colon [0,T] \to \R^d$ be continuous and of bounded variation, i.e.\ there exists finite signed measures $\mu^i\colon \borel([0,T]) \to \R$, such that for all $t\in[0,T]$, $\mu^i([0,t]) = X^i(t)$.
Then the \textit{signature} is a two-parameter path $ \sig{X}{}\colon \Delta_T \to T((\R^d))$, where for every $(s,t)\in\Delta_T$, $k\in\N$,
\begin{align*}
	\proj_k \sig{X}{s,t} = \sum_{\substack{w = (w_1, \dots, w_k),\\
			w_j\in\A}} \langle \sig{X}{s,t}, e_w\rangle e_w,
\end{align*}
with
\begin{align}\label{defi:signature}
	\langle \sig{X}{s,t}, e_w\rangle := \int_s^t \int_s^{s_k} \dots \int_s^{s_2} \dd X^{w_1}(s_1) \dots \dd X^{w_k}(s_k).
\end{align}
The \textit{symmetric part} $\sym{T}$ of a $k$-tensor $T$ is given via 
\begin{align}\label{defi:sympart}
	\langle \sym{T}, e_w\rangle = \frac{1}{k!} \sum_{\sigma\in \mathfrak{S}_k} \langle T, e_{(w_{\sigma1}, \dots, w_{\sigma k})}\rangle,
\end{align}  
where $\mathfrak{S}_k$ denote the permutation group of degree $k$, see \parencite[Chapter 4, $\oint$5]{kostrikinLinearAlgebraGeometry1997} for an introduction.
We next define a commutative product on tensors indexed by words. Let  $m, k_1,\dots, k_m\in\N$ and $k= \sum_{j=1}^m k_j$. The \textit{shuffles} $\mathrm{sh}(k_1,\dots, k_m)$ of words of length $ k_1,\dots, k_m$ are those permutations $\sigma\in\mathfrak{S}_k$ such that $\sigma 1 < \dots < \sigma k_1$, $\sigma (k_1+1) < \dots <  \sigma (k_1+k_2)$ and so on. Then for a word $w$ of length $k$, we define
\begin{align*}
e_{(w_1, \dots, w_{k_1})} \shuffle \dots \shuffle e_{(w_{k - k_m + 1}, \dots, w_{k})} = \sum_{\sigma\in\mathrm{sh}(k_1, \dots, k_m)} e_{(w_{\sigma 1}, \dots, w_{\sigma k})}.
\end{align*}
Note that for letters $w_1, \dots, w_k\in\A$, this reduces to
\begin{align}\label{shuffleletters}
	 e_{w_1}\shuffle \dots \shuffle e_{w_k} = \sum_{ \sigma\in\mathfrak{S}_k} e_{(w_{\sigma 1},\dots, w_{\sigma k})}.
\end{align} 
It is easy to check that the signature has the remarkable property that for two words $w$ and $u$ it holds that
\begin{align}\label{shufflesignature}
	\langle \sig{X}{s,t}, e_w\shuffle e_u \rangle = \langle \sig{X}{s,t}, e_w\rangle \langle \sig{X}{s,t}, e_u\rangle, 
\end{align}
compare \parencite[Exercise 2.2]{RoughBook}. We deduce that the symmetric part of the $k$th level signature is 
\begin{align}\label{symsignature}
	\sym{\proj_k \sig{X}{s,t}} = \frac{1}{k!} (X(t) - X(s))^{\otimes k},
\end{align}
since it follows from \eqref{defi:sympart}, \eqref{shuffleletters}, \eqref{shufflesignature} and \eqref{defi:signature} for a word $w$ of length $k$ that
\begin{align*}
	\langle \sym{\proj_k \sig{X}{s,t}}, e_w \rangle 
	&=\frac{1}{k!} \sum_{\sigma\in \mathfrak{S}_k} \langle \sig{X}{s,t}, e_{(w_{\sigma1}, \dots, w_{\sigma k})}\rangle
	= \frac{1}{k!} \langle \sig{X}{s,t}, e_{w_1}\shuffle \dots \shuffle e_{w_k}\rangle\\
	&= \frac{1}{k!} \prod_{j=1}^k \langle \sig{X}{s,t}, e_{w_j} \rangle = \frac{1}{k!} \prod_{j=1}^k (X^{w_j}(t) - X^{w_j}(s)).
\end{align*}
Finally we point out that \parencite[Proposition 3.5]{frizUnifiedSignatureCumulants2021} shows that the symmetric part of the signature satisfies 
\begin{align}\label{ddsymsig}
\sym{\proj_k \sig{X}{s,t}} = \sym{\int_s^t \sym{\proj_{k-1} \sig{X}{s,r}} \otimes \dd X(r)}.
\end{align}
\add{Proof basically just $\dd\sig{X}{} = \sig{X}{}\otimes \dd X$ and so
\[\sym{\int_s^t \sym{\proj_{k-1} \sig{X}{s,r}} \otimes \dd X(r)} = \sym{\int_s^t \proj_{k-1} \sig{X}{s,r}\otimes \dd X(r)} = \sym{\proj_{k} \sig{X}{s,t}}. 
\]}

\subsection{Taylor Formula for Causal Functionals}

The first result establishes a Taylor formula in terms of path-dependent derivatives for paths of bounded variation. It is based on the Taylor expansion of one-dimensional and piecewise constant paths $X$ that is used in \parencite{BallyVlad2016Sibp, contPathwiseIntegrationChange2019} to prove the functional It\^o formula with F\"ollmer integrals. It will prove very useful to have an explicit representation of the remainder.

\begin{theo}[Taylor Formula for Functionals of Bounded Variation Paths]\label{theo:taylor}
Let $n, d\in\N$ and $F\in\mathbb{C}^{1,n}_b$, such that for $k=1, \dots, n-1$, $\nabla^k F\in\mathbb{C}^{1,1}_b$. Then it holds for every path $X\colon [0,T]\to \R^d$ that is continuous and of bounded variation and every $(s,t)\in\Delta_T$, that 
\begin{align}
F(t, X) - F(s, X)&= \sum_{k=0}^{n-1} \frac{1}{k!}\int_s^t \la D\nabla^k F(u, X), (X(t) - X(u))^{\otimes k} \ra\dd u + \sum_{k=1}^{n-1} \frac{1}{k!} \la\nabla^k F(s, X),(X(t) - X(s))^{\otimes k} \ra\notag\\
&\quad + \frac{1}{(n-1)!} \int_s^t \langle \nabla^n F(u, X),(X(t) - X(u))^{\otimes n-1}\otimes\dd X(u)\rangle .\label{theo:taylor:remainder}
\end{align}
\end{theo}

\begin{proof}
The proof is by induction on $n$.

Note that in the case that $d=1$, the result follows for  $n=1$ from \parencite[Theorem 1.10]{contPathwiseIntegrationChange2019} applied with $p=2$ and $\qv{X}{}^2 = 0$ with higher regularity assumption on $F$. For consistency we give the start of the induction with minor changes due to $d\ge 1$ and $X$ more regular.

For $n=1$ let $(\partition)$ be a sequence of partitions of $[s,t]$ with $\vert\partition\vert \to 0$. We consider the piecewise constant approximation of $X$ on $[s,t]$:
\begin{align*}
X^{\partition}(u) = X(u)\one_{[0,s)}(u) + \sum_{[t_j, t_{j+1}]\in\partition} X(t_{j+1})\one_{[t_j, t_{j+1})} + X(t)\one_{[t, T]}. 
\end{align*}
Since $X^{\partition} \to X$ uniformly and $X^{\partition}_{t-} = X^{\partition}_t$, $X^{\partition}_{s-} = X_s$, it holds 
\begin{align}
F(t, X) - F(s,X) = \lim_{\vert\partition\vert\to 0} \sum_{[t_j, t_{j+1}]\in\partition} F(t_{j+1}, X^{\partition}_{t_{j+1}-}) - F(t_j,X^{\partition}_{t_j-}).\label{theo:taylor:eq:1}
\end{align} 
Noting that $X^{\partition}_{t_j} = X^{\partition}_{t_{j+1}-}$ on $[0, t_{j+1}]$ and $X^{\partition}_{t_j} = X^{\partition}_{t_j-} + (X(t_{j+1}) - X({t_j}))\one_{[t_j, T]}$ we decompose the difference into the time and space perturbation,
\begin{align}
&F(t_{j+1}, X^{\partition}_{t_{j+1}-}) - F(t_j,X^{\partition}_{t_j-})\notag\\
&= F(t_{j+1}, X^{\partition}_{t_j}) - F(t_j, X^{\partition}_{t_j})  + F(t_j, X^{\partition}_{t_j-} + (X(t_{j+1}) - X({t_j}))\one_{[t_j, T]}) - F(t_j,X^{\partition}_{t_j-}).\label{theo:taylor:eq:2}
\end{align}
By construction of the causal time derivative $[t_j, t_{j+1})\ni u \mapsto F(u, X^{\partition}_{t_j})$ is right-differentiable with Riemann integrable derivatives, thus by the fundamental theorem of calculus, cf. \parencite{botskoStrongerVersionsFundamental1986}, it holds 
\begin{align*}
\sum_{[t_j, t_{j+1}]\in\partition} F(t_{j+1}, X^{\partition}_{t_j}) - F(t_j, X^{\partition}_{t_j})
= \sum_{[t_j, t_{j+1}]\in\partition} \int_{t_j}^{t_{j+1}} D F(u, X^{\partition}_{t_j}) \dd u.
\end{align*}
Since for every $u\in[s,t]$, $\sum_{[t_j, t_{j+1}]\in\partition} F(u, X^{\partition}_{t_j}) \one_{[t_j, t_{j+1})}(u) = F(u, X^{\partition}_u) \to F(u, X_u)$ as $\vert\partition\vert\to 0$ and $DF$ bounded, the last expression converges to $\int_s^t  DF(u, X) \dd u$. Similarly it follows for the space perturbation in \eqref{theo:taylor:eq:2} that
\begin{align*}
&F(t_j, X^{\partition}_{t_j-} + (X(t_{j+1}) - X({t_j}))\one_{[t_j, T]}) - F(t_j,X^{\partition}_{t_j-})\\
&= \int_0^1 \langle\nabla F(t_j, X^{\partition}_{t_j-} + \lambda(X(t_{j+1}) - X({t_j})\one_{[t_j, T]}), (X(t_{j+1}) - X({t_j}))\rangle \dd\lambda\\
&=: \langle\nabla F(t_j,X^{\partition}_{t_j-}), (X(t_{j+1}) - X({t_j}))\rangle + R_j.
\end{align*}
For $i=1,\dots, d$ let now $\mu^i$ denote measures of bounded variation related to component $X^i$. Since $\sum_{[t_j, t_{j+1}]\in\partition} \partial_i F(t_j,X^{\partition}_{t_j-})\one_{[t_j, t_{j+1})}(u) \to \partial_i F(u, X)$ as $\vert\partition\vert\to 0$ and $\partial_i F$ bounded, it follows that
\begin{align*}
\sum_{[t_j, t_{j+1}]\in\partition} \partial_i F(t_j,X^{\partition}_{t_j-}) (X^i(t_{j+1}) - X^i({t_j})) 
&= \int_s^t \sum_{[t_j, t_{j+1}]\in\partition} \partial_i F(t_j,X^{\partition}_{t_j-})\one_{[t_j, t_{j+1})}(u) \dd\mu^i(u)\\
& \to \int_s^t \partial_i F(u, X) \dd \mu^i(u)
= \int_s^t \partial_i F(u, X) \dd X^i(u).
\end{align*}
Moreover using that the images of $(\id,X^{\partition})$ are compact in $[0,T]\times D$, we may assume that $\nabla F$ is compactly supported and therefore uniformly continuous. 
\add{
For every $n\in\N$, $K_n = \{(t_j, X^{\partition}_{t_j-}+ \lambda(X(t_{j+1}) - X(t_j))\one_{[t_j, T]})\colon j = 0, \dots ,n-1,\  \lambda\in[0,1]\}$ is compact: Any sequence in $K_n$ corresponds to a sequence $t_m$ in the finite set $\{t_j\colon j=0,\dots, n-1\}$ and $(\lambda_m)\subset[0,1]$. Then there are note relabeled subsequences with $t_m \to t_j^*$ and $\lambda_m \to\lambda*\in[0,1]$. Clearly for large $m$, $t_m = t_j^*$. Then trivially \[(t_m, X^{\partition}_{t_m-}+ \lambda_m(X(t_{m+1}) - X(t_m)\one_{[t_m, T]}) = (t_j^*, X^{\partition}_{t_j^*-}+ \lambda_m(X(t_{j+1}^*) - X(t_j^*)\one_{[t_j^*, T]}) \to (t_j^*, X^{\partition}_{t_j^*-}+ \lambda^*(X(t_{j+1}^*) - X(t_j^*))\one_{[t_j^*, T]})\in K_n.\]}
Hence the remainders 
\begin{align*}
R_j = \int_0^1 \nabla F(t_j, X^{\partition}_{t_j-} + \lambda(X(t_{j+1}) - X({t_j}))\one_{[t_j, T]}) - \nabla F(t_j,X^{\partition}_{t_j-}) \dd\lambda \cdot (X(t_{j+1}) - X({t_j}))
\end{align*}
satisfy 
\begin{align*}
&\sum_{[t_j, t_{j+1}]\in\partition} \vert R_j\vert 
\le C(\vert \nabla F\vert_\infty, \partition) \sum_{[t_j, t_{j+1}]\in\partition} \vert X(t_{j+1}) - X({t_j})\vert
\le C(\vert \nabla F\vert_\infty, \partition)\vert\mu\vert([s,t]) 
\end{align*}
where $\vert\mu\vert$ denotes the total variation of $\mu$ and $C(\vert \nabla F\vert_\infty, \partition)\to 0$ as $\vert\partition\vert\to\infty$.\\

For $n \to n+1$ we apply the previous result componentwise to get 
\begin{align*}
\nabla^n F(u, X) - \nabla^n F(s, X) = \int_s^u D\nabla^n F(r, X) \dd r + \int_s^u \langle \nabla^{n+1} F(r, X), \dd X(r) \rangle (\cdot).
\end{align*}
Plugging that into the remainder \eqref{theo:taylor:remainder} and using Fubini, it follows
\begin{align}
&\int_s^t  \langle \nabla^n F(u,  X), ( X(t) -  X(u))^{ \otimes n-1} \otimes \dd X(u)\rangle\notag\\
\begin{split}
&= \int_s^t \la D\nabla^n F(r,X),  \int_r^t (X(t) -  X(u))^{\otimes n-1}\otimes \dd X(u)\ra \dd r \\
&\quad  + \int_s^t \la\nabla^{n+1} F(r,X),\int_r^t (X(t) -  X(u))^{\otimes n-1}\otimes\dd X(u) \otimes\dd X(r)\ra\\
&\quad + \la\nabla^n F(s,  X),\int_s^t( X(t) -  X(u))^{\otimes n-1}\otimes\dd X(u)\ra
\end{split}\label{theo:taylor:eq:2a}
\end{align}
For every $r\in[s,t]$ the function $g(h):= F(r, X_r + h\one_{[r,T]})$ is $(n+1)$-times continuously differentiable in zero by assumption. Thus Schwarz' lemma shows that the causal space derivative $\nabla^{n+1} F(r, X) = \nabla^{n+1} g(0)$ is a symmetric tensor (i.e., $\sym{\nabla^{n+1} F(r, X)} = \nabla^{n+1} F(r, X)$). 
Consequently,
\begin{align}\label{theo:taylor:eq:3}
&\la\nabla^n F(s,  X),\int_s^t( X(t) -  X(u))^{\otimes n-1}\otimes\dd X(u)\ra= \la\nabla^n F(s,  X), \sym{\int_s^t( X(t) -  X(u))^{\otimes n-1}\otimes\dd X(u)}\ra.
\end{align}
It holds that
\begin{align}
&\sym{\int_s^t( X(t) -  X(u))^{\otimes n-1}\otimes\dd X(u)}\notag\\
&= \sum_{l=0}^{n-1} (-1)^{n-l-1} \binom{n-1}{l} \sym{(X(t)-X(s))^{\otimes l} \otimes \int_s^t (X(u) - X(s))^{\otimes n-l-1}\otimes \dd X(u)}\label{theo:taylor:eq:4}.
\end{align} 
Recalling \eqref{symsignature} and \eqref{ddsymsig} we deduce
\begin{align*}
\frac{1}{(n-l-1)!}\sym{\int_s^t (X(u) - X(s))^{\otimes n-l-1}\otimes \dd X(u)}
= \frac{1}{(n-l)!} (X(t)- X(s))^{\otimes n-l}.
\end{align*}
And plugging that into \eqref{theo:taylor:eq:4} yields
\begin{align*}
\sym{\int_s^t( X(t) -  X(u))^{\otimes n-1}\otimes\dd X(u)} &= (X(t)-X(s))^{\otimes n} \sum_{l=0}^{n-1} (-1)^{n-l-1} \binom{n-1}{l} \frac{(n-l-1)!}{(n-l)!}\\
&=\frac{(n-1)!}{n!} (X(t)-X(s))^{\otimes n},
\end{align*}
since $\sum_{l=0}^{n-1} (-1)^{n-l-1} \binom{n}{l} = 1$. Using similar arguments for the inner product with $D\nabla^n F$ and $\nabla^{n+1} F$ in \eqref{theo:taylor:eq:2a} we conclude that
\begin{align*}
&\frac{n!}{(n-1)!}\int_s^t  \langle \nabla^n F(u,  X), ( X(t) -  X(u))^{ \otimes n-1} \otimes \dd X(u)\rangle\\
&=\int_s^t \la D\nabla^n F(r,X), ( X(t) -  X(r))^{\otimes n} \ra\dd r +\int_s^t \la \nabla^{n+1} F(r,  X),(X(t) -  X(r))^{\otimes n} \otimes\dd X(r)\\
&\quad  + \la \nabla^n F(s,X),(X(t)-X(s))^{\otimes n} \ra.
\end{align*}
\end{proof}

\begin{rem} Note that it is sufficient that the functional $F$ and its causal derivatives are continuous. As seen in the proof the images $(\id, X^{\partition})$ lie in a compact subset of $[0,T]\times D$ if $X$ is continuous. Then any continuous functional restricted to this compact metric space is uniformly continuous and bounded.  
\end{rem}

The next corollary is a generalization of \parencite[Lemma  2.2]{ananovaPathwiseIntegrationRespect2017}. 
It estimates the error of a lower order Taylor approximation of $F$ composed with an $\alpha$-H\"older continuous path $X\colon [0,T]\to \R^m$. The reader may notice that the previously mentioned result \parencite[Theorem 1.10]{contPathwiseIntegrationChange2019} can be applied to less regular paths using F\"ollmer integrals. But since we want to estimate the remainder \eqref{theo:taylor:remainder} for the next result, we prefer to use integrals against paths of bounded variation.

\begin{cor}[Taylor Approximation of Causal Functionals of H\"older Continuous Paths]\label{cor:remainderestimate}
Let $X\colon [0,T] \to \R^d$ be $\alpha$-H\"older continuous for some $\alpha\in(0,1)$ and $F$ as in Theorem \ref{theo:taylor}. Assume additionally that $F$ and $D F$ are Lipschitz continuous for fixed times with bounded Lipschitz constants. Then it holds for every $(s,t)\in\Delta_T$ with $\vert t-s\vert \le 1$ and $0\le l\le n-1$ that
\begin{align}
\begin{split}\label{cor:remainderestimate:eq}
\Bigg\vert F(t, X) - \int_s^t DF(r,X)\dd r- \sum_{k=0}^l \frac{1}{k!}\la \nabla^k F(s, X),(X(t) - X(s))^{\otimes k} \ra \Bigg\vert\\
\qquad \lesssim \vert t-s\vert^{\alpha + (n-1)\alpha^2} + \vert t-s\vert^{1+\alpha} + \vert t-s\vert^{(l+1)\alpha},
\end{split}
\end{align}
with a constant depending on $n$, $l$, $\vert F\vert_\infty$, $\vert D\nabla^k F\vert_\infty$ and $\vert\nabla^{k+1}F\vert_\infty$ for $k=l+1,\dots, n-1$ as well as $\sup_{r\in[s,t]}\{ Lip(F(r, \cdot), DF(r, \cdot))\}$ and $\vert X\vert_\alpha$.
\end{cor}

\begin{rem}
We point out the differences of \eqref{cor:remainderestimate:eq} to a typical Taylor approximation. As usual the exponent $(l+1)\alpha$ connected to the $l$ space derivatives used in the approximation. The appearance of $(1+\alpha)$ is due to the path-dependent time derivatives. And finally $\alpha + (n-1)\alpha^2$ due to an approximation of $X$ by piecewise constant paths.
\end{rem}

\begin{proof}
Let $\partition$ be a partition of $[s,t]$ whose subintervals are all of length $\vert\partition\vert$.
Consider a piecewise linear approximation $X^{\partition}$ of $X$ on $[s,t]$ such that $X^{\partition}_s = X_s$, and for every $[u,v]\in\partition$ it holds $X^{\partition}(u)= X(u)$ and $X^{\partition}(v) = X(v)$ and in between $X^{\partition}$ is linearly interpolated. Then $X^{\partition}$ is continuous and on $[s,t]$ of bounded variation. Hence the previous theorem shows that
\begin{align}
& F(t, X^{\partition}) - \int_s^t DF(r,X^{\partition})\dd r- \sum_{k=0}^l \frac{1}{k!}\la \nabla^k F(s, X^{\partition}),(X^{\partition}(t) - X^{\partition}(s))^{\otimes k} \ra\notag\\
\begin{split}
&= \sum_{k=1}^{n-1} \frac{1}{k!}\int_s^t \la D\nabla^k F(r, X^{\partition}), (X^{\partition}(t) - X^{\partition}(r))^{\otimes k}\ra \dd r+ \sum_{k=l+1}^{n-1} \frac{1}{k!} \la\nabla^k F(s, X^{\partition}),(X^{\partition}(t) - X^{\partition}(s))^{\otimes k}\ra\\
&\quad  + \frac{1}{(n-1)!} \int_s^t \langle \nabla^n F(r, X^{\partition}),(X^{\partition}(t) - X^{\partition}(r))^{\otimes n-1}\otimes\dd X^{\partition}(r)\rangle. 
\end{split}\label{cor:remainderestimate:eq:1}
\end{align}
Note that $X^{\partition}$ is also $\alpha$-H\"older continuous  with $\vert X^{\partition}\vert_\alpha \lesssim \vert X\vert_\alpha$
\add{ $u\le v\in [t_j, t_{j+1}]$ it holds $X^{\partition}(v) - X^{\partition}(u)= \frac{v-u}{t_{j+1} - t_j} (X(t_{j+1}) - X(t_j))$, so \[\frac{\vert X^{\partition}(v) - X^{\partition}(u)\vert}{\vert v-u\vert^\alpha} \lesssim_{\vert X\vert_\alpha} \lb \frac{v-u}{t_{j+1} - t_j}\rb^{1-\alpha} \le 1.\]
And for $u\in[t_j, t_{t_j+1}]$, $v\in[t_i, t_{i+1}]$ with $j<i$ it follows that \[\frac{\vert X^{\partition}(v) - X^{\partition}(u)\vert}{\vert v-u\vert^\alpha} \lesssim_{\vert X\vert_\alpha} \left(\frac{t_{j+1}-u}{v-u}\right)^\alpha + \left(\frac{t_i - t_{j+1}}{v-u}\right)^\alpha +  \left(\frac{v - t_i}{v-u}\right)^\alpha \le 3.\] } 
and that for every $[u,v]\in\partition$ it holds on $(u,v)$,
\begin{align}\label{cor:remainderestimate:eq:1a}
\left \vert \frac{\dd X^{\partition}}{\dd r}\right\vert = \left \vert\frac{X(v) - X(u)}{v- u}\right\vert 
\lesssim_{\vert X\vert_\alpha} \vert v-u \vert^{\alpha - 1} = \vert\partition\vert^{\alpha - 1}.
\end{align}
In the case $l\le n-2$, it follows that \eqref{cor:remainderestimate:eq:1} is bounded by 
\begin{align}
	\begin{split}\label{cor:remainderestimate:eq:2}
&\sum_{k=1}^{n-1} \frac{1}{k!} \vert D\nabla^k F\vert_\infty \vert X^{\partition}\vert_\alpha^k\vert t - s\vert^{k\alpha} \vert t-s\vert  + \sum_{k=l+1}^{n-1} \frac{1}{k!} \vert \nabla^k F\vert_\infty \vert X\vert_\alpha^k\vert t-s\vert^{k\alpha} \\
&\quad + \frac{1}{(n-1)!}\vert \nabla^n F\vert_\infty \vert X^{\partition}\vert_\alpha^{(n-1)} \vert t-s\vert^{(n-1)\alpha}\vert\partition\vert^{\alpha - 1} \vert t-s\vert\\
&\lesssim_{n,l, \vert D\nabla^k F\vert_\infty, \vert X\vert_\alpha, \vert \nabla^k F\vert_\infty} \vert t-s\vert^{1+\alpha} + \vert t-s\vert^{(l+1)\alpha} + \vert t-s\vert^{1+ (n-1)\alpha}\vert\partition\vert^{\alpha - 1},
	\end{split}
\end{align}
where we picked in both sums the smallest exponent and restricted to the case $\vert t-s\vert\le 1$.
Since $F(s,X^{\partition}) = F(s,X)$ and $X^{\partition}(t) -X^{\partition}(s) = X(t) - X(s)$ it follows that
\begin{align}
&F(t, X) - \int_s^t DF(r,X)\dd r- \sum_{k=0}^l \frac{1}{k!}\la\nabla^k F(s, X),(X(t) - X(s))^{\otimes k} \ra\notag\\
\begin{split}
&= F(t, X)-F(t, X^{\partition}) - \int_s^t DF(r, X) - DF(r, X^{\partition}) \dd r\\
&\quad +  F(t, X^{\partition}) - \int_s^t DF(r,X^{\partition})\dd r- \sum_{k=0}^l \frac{1}{k!}\la \nabla^k F(s, X^{\partition}),(X^{\partition}(t) - X^{\partition}(s))^{\otimes k} \ra.\label{cor:remainderestimate:eq:3}
\end{split}
\end{align}
Since $F$ and $DF$ are Lipschitz continuous for fixed times it holds that 
\begin{align*}
\vert F(t, X)-F(t, X^{\partition}) \vert \lesssim_{Lip(F(t,\cdot))} \vert X - X^{\partition} \vert_\infty \lesssim_{\vert X\vert_\alpha} \vert \partition\vert^{\alpha}
\end{align*}
and similar
\begin{align*}
\left\vert \int_s^t DF(r, X) - DF(r, X^{\partition}) \dd r \right\vert \lesssim_{\sup_{r\in[s,t]} Lip(DF(r,\cdot))} \vert X - X^{\partition}\vert_\infty \vert t-s\vert \lesssim_{\vert X\vert_\alpha} \vert t-s\vert \vert \partition\vert^{\alpha}.
\end{align*}
Together with estimate \eqref{cor:remainderestimate:eq:2}, we deduced that \eqref{cor:remainderestimate:eq:3} is bounded by a constant depending on $n$, $l$, $Lip(F)$,  $\sup_{r\in[s,t]} Lip(DF(r, \cdot)$, $\vert X\vert_\alpha$, $\vert D\nabla^k F\vert_\infty$, $\vert \nabla^k F\vert_\infty$ times
\begin{align*}
 \vert \partition\vert^{\alpha} + \vert t-s\vert\vert \partition\vert^{\alpha}  + \vert t-s\vert^{1+\alpha} +\vert t-s\vert^{(l+1)\alpha} + \vert t-s\vert^{1+ (n-1)\alpha}\vert\partition\vert^{\alpha - 1}.
\end{align*}
Optimizing the choice of $\vert\partition\vert$, by balancing $\vert \partition\vert^{\alpha} \approx \vert t-s\vert^{1+ (n-1)\alpha}\vert\partition\vert^{\alpha - 1}$, i.e.\ $\vert\partition\vert\approx\vert t-s\vert^{1+(n-1)\alpha}$ we obtain the assertion for $0\le l\le n-2$. Finally note that for $l=n-1$, the second sum in \eqref{cor:remainderestimate:eq:2} is empty, so the RHS is simply $\vert t-s\vert^{1+\alpha}+\vert t-s\vert^{1+ (n-1)\alpha}\vert\partition\vert^{\alpha - 1}$. Nevertheless there is nothing wrong in writing $\vert t-s\vert^{n\alpha}$ in the assertion \eqref{cor:remainderestimate:eq}, since $\alpha + (n-1)\alpha^2 < n\alpha$. 
\add{$\vert\partition\vert\approx\vert t-s\vert^{1+(n-1)\alpha}$, in the sense that the number of subintervals $m=\vert t-s\vert/\vert\partition\vert\in\N$ satisfies \[  \vert t-s\vert^{-(n-1)\alpha} \le m \le 2\vert t-s\vert^{-(n-1)\alpha},\] possible because $\vert t-s\vert^{-(n-1)\alpha}\ge 1$. Then \[\vert\partition\vert^\alpha \le \vert t-s\vert^{\alpha + (n-1)\alpha^2}, \quad \vert\partition\vert^{\alpha-1} \le 2^{1-\alpha} \vert t-s\vert^{-(1+(n-1)\alpha) + \alpha+(n-1)\alpha^2}.\]}
\end{proof}

\section{Rough Functional It\^{o} Formula}\label{sec:RoughIto}

Let $\mathrm{Sym}_k(\R^d)$ denote the subspace of symmetric $k$-tensors on $\R^d$ and $\mathbb{S}_n(\R^d) = \oplus_{k=0}^n \mathrm{Sym}_k(\R^d)$ for their direct sum.
Throughout this section we denote the point evaluation of two-parameter paths $\Xi$ by $\Xi_{s,t} = \Xi(s,t)$. 
For an $\alpha$-H\"older continuous path $X$ and $(s,t)\in\Delta_T$, we set $\X^0_{s,t}:=1$ and for $k\ge1$,
\begin{align}\label{notationlevel}
\X^k_{s,t} := \frac{1}{k!}(X(t) - X(s))^{\otimes k}.
\end{align}
Further we write for their collection $\Xb := (\X^0, \X^1, \dots)$. It was shown in \parencite[Definition 4.6, Lemma 4.7]{contPathwiseIntegrationChange2019} that for every $k\ge 1$, $\X^k\colon \Delta_T \to \mathrm{Sym}_k(\R^d)$ is a $k\alpha$-H\"older continuous two-parameter path and a \textit{reduced Chen relation} holds: For every $(s,u), (u,t)\in\Delta_T$,
\begin{align}\label{reducedChen}
\Xb_{s,t} = \sym{\Xb_{s,u} \otimes \Xb_{u,t}}.
\end{align}

\begin{defi}
	Let $X$ be $\alpha$-H\"older continuous for $\alpha\in (0,1)$ and $\beta\in(0,1)^{d+1}$. A path $Y=(Y^0, Y^1,\dots, Y^n)\in C([0,T],\mathbb{S}_n(\R^d))$ is $\beta$-controlled by $X$ if there exists $C>0$, such that for all $(s,t)\in\Delta_T$,
	\begin{align*}
		\sum_{l=1}^n \vert R_{s,t}^{X,k}\vert^{\lb\sum_{k=1}^n \beta_k -\sum_{k=n-l+1}^n\beta_k\rb^{-1}} \le C\vert t-s\vert,
	\end{align*}
	where the $kth$ remainder is
	\begin{align*}
		R_{s,t}^{X, k} := Y^k(t) - \sum_{l=k}^n \langle Y^l(s), \X_{s,t}^{l-k}\rangle.
	\end{align*}
\end{defi}

\begin{lem}\label{lem:YXcontrolled}
	Let $X$ be $\alpha$-H\"older continuous for $\alpha\in (0,1)$ and $n$ be the smallest natural number such that $2\alpha + (n-1)\alpha^2>1$. Let further $F$ be as in Corollary \ref{cor:remainderestimate} to parameter $n+1$ and assume additionally that for each $k=1, \dots, n$, $\nabla^k F$ and $D\nabla^k F$ are also Lipschitz continuous for fixed times with bounded Lipschitz constants. Then $Y\in C([0,T],\mathbb{S}_n(\R^d))$ given by $Y^0\equiv1$ and for $s\in[0,T]$,
	\begin{align*}
		Y^k(s) = \nabla^k F(s, X)
	\end{align*}
	is $(1,\alpha, \alpha^2,\dots, \alpha^2)$-controlled by $\Xb$.
\end{lem}

\begin{proof}
	In the proof of Theorem \ref{theo:taylor} it was discussed that $Y^k\in\mathrm{Sym}_k(\R^d)$. 
	For $k=1,\dots,n$ Corollary \ref{cor:remainderestimate} applied for $\nabla^k F\in\mathbb{C}_b^{1,(n+1-k)}$ with $l=n-k$ and $(s,t)\in\Delta_T$ with $\vert t-s\vert\le 1$, shows that
	\begin{align*}
		\vert R_{s,t}^{X, k}\vert \lesssim \vert u-s\vert^{\alpha + (n-k)\alpha^2}.
	\end{align*}
	The other terms in \eqref{cor:remainderestimate:eq} don't appear since $\alpha\in(0,1)$ and the minimality of $n$ imply $\alpha + (n-k)\alpha^2 < (n-k+1) \alpha$ and $\alpha + (n-k)\alpha^2<1$. 	\add{$1\ge\alpha + (n-k)\alpha^2$, if not for some $k$, then $1 <\alpha +(n-1)\alpha^2$. But then $2\alpha + (n-2)\alpha^2 > \alpha + (n-1)\alpha^2 > 1$, which is a contradiction to $n$ smallest natural number satisfying $2\alpha + (n-1) \alpha^2 > 1$.}
	To iterate the bound similar to \parencite[Exercise 4.5]{RoughBook}, we need to take into account that the remainders are not additive, but almost.\\
	Assume more generally we have a bound for intervals not longer than $h>0$. Let $\vert t-s\vert >h$, set $t_i =(s+ih)\wedge t$ and $N\le \vert t-s\vert/h +1$ the number of subintervals $[t_j,t_{j+1}]$. It holds that
	\begin{align*}
		\vert R_{s,t}^{X,k}\vert \le \Big\vert R_{s,t}^{X,k}-\sum_{j=0}^{N-1} R_{t_j,t_{j+1}}^{X,k}\Big\vert + \sum_{j=0}^{N-1} \vert R_{t_j,t_{j+1}}^{X,k}\vert
	\end{align*}
	with $\vert R_{t_j,t_{j+1}}^{X,k}\vert \lesssim h^{\alpha+(n-k)\alpha^2}$ and 
	\begin{align*}
		R_{s,t}^{X,k}-\sum_{j=0}^{N-1} R_{t_j,t_{j+1}}^{X,k} = \sum_{l=k+1}^n \lb \sum_{j=0}^{N-1} \la Y^k(t_j), \X_{t_j, t_{j+1}}^{l-k}\ra -\la Y^k(s),\X^{l-k}_{s,t}\ra\rb.
	\end{align*}
	On the one hand by Corollary \ref{cor:remainderestimate} applied for $\nabla^k F\in\mathbb{C}_b^{1,1}$ with $l=0$,
	\begin{align*}
		\vert \nabla^k F(t_j, X)-\nabla^k F(s, X)\vert \lesssim \vert t_j-s\vert^\alpha.
	\end{align*}
	Thus
	\begin{align*}
		\vert \la Y^k(t_j)-Y^k(s), \X_{t_j, t_{j+1}}^{l-k}\ra\vert 
		= \vert \la \nabla^k F(t_j, X)-\nabla^k F(s, X), \X_{t_j, t_{j+1}}^{l-k}\ra\vert \lesssim \vert t_j-s\vert^{\alpha} h^{(l-k)\alpha}.
	\end{align*}
	On the other hand consider the differences $\sum_{j=0}^{N-1} \X^{l-k}_{t_j,t_{j+1}} - \X^{l-k}_{s,t}$. For $l=k+1$, $\X^{l-k} =\X^1$ is additive, so $\sum_{j=0}^{N-1} \X^1_{t_j,t_{j+1}} - \X^1_{s,t}=0$. Moreover iterating the reduced Chen relation \eqref{reducedChen} and using that the symmetric tensor product is associative, it holds that
	\begin{align*}
		\Xb_{s,t} = \sym{\otimes_{j=0}^{N-1} \Xb_{t_j, t_{j+1}}}.
	\end{align*}
	This implies for $l>k+1$ that
	\begin{align*}
		&\sum_{j=0}^{N-1} \X^{l-k}_{t_j,t_{j+1}} - \X^{l-k}_{s,t}
		= \sum_{j=0}^{N-1} \X^{l-k}_{t_j,t_{j+1}} - \proj_{l-k}\sym{\otimes_{j=0}^{N-1} \Xb_{t_j, t_{j+1}}}\\
		&= - \sum_{\substack{k_0+\dots+k_{N_1} = l-k\\ 0\le k_j< l-k}} \sym{\X_{t_0, t_1}^{k_0}\otimes \dots\otimes \X^{k_{N-1}}_{t_{N-1}, t_N}}.
	\end{align*}
	It follows for $k=l+2,\dots,n$ that 
	\begin{align*}
		&\Big\vert \la \nabla^k F(s,X), \sum_{j=0}^{N-1} \X^{l-k}_{t_j,t_{j+1}} - \X^{l-k}_{s,t}\ra \Big\vert
		\lesssim \sum_{\substack{k_1+\dots+k_N = l-k\\ 0\le k_j< l-k}} \prod_{j=0}^{N-1} \vert \X_{t_j, t_{j+1}}^{k_j}\vert\\
		&\le \sum_{\substack{k_1+\dots+k_N = l-k\\ 0\le k_j< l-k}} \prod_{j=0}^{N-1} \frac{1}{k_j!} h^{k_j\alpha} = \frac{1}{(l-k)!} (N^{l-k}-N)h^{(l-k)\alpha}
	\end{align*}
	by the multinomial theorem. All together we showed that
	\begin{align*}
		\vert R_{s,t}^{X,k}\vert
		\lesssim \sum_{l=k+1}^n N\vert t-s\vert^\alpha h^{(l-k)\alpha} + \sum_{l=k+2}^n \frac{1}{(l-k)!} N^{l-k-1}(N-1)h^{(l-k)\alpha} + Nh^{\alpha+(n-k)\alpha^2}.
	\end{align*}
	Using that $N\le\vert t-s\vert/h +1 = (\vert t-s\vert +h)h^{-1}< 2\vert t-s\vert h^{-1}$ yields
	\begin{align*}
		\vert R_{s,t}^{X,k}\vert\lesssim \vert t-s\vert^{1+\alpha} \sum_{l=k+1}^n
		h^{(l-k)\alpha-1} + \sum_{l=k+2}^n	\vert t-s\vert^{l-k} h^{(l-k)\alpha-1} + \vert t-s\vert h^{\alpha+(n-k)\alpha^2-1}.
	\end{align*}
	Dividing by $\vert t-s\vert^{\alpha+(n-k)\alpha^2}$ and recalling that $\alpha+(n-k)\alpha^2<1$, it follows that together with $h=1$, that $R^{X,k}$ is $\alpha+(n-k)\alpha^2$-H\"older continuous. 
\end{proof}

\begin{theo}[Rough Functional It\^{o} Formula]\label{theo:roughfuncInt}
	Let $X$ be $\alpha$-H\"older continuous for $\alpha\in (0,1)$ and $n$ be the smallest natural number such that $2\alpha + (n-1)\alpha^2>1$. Let further $F$ be as in Lemma \ref{lem:YXcontrolled}.
	Then 
	\begin{align}\label{theo:roughfuncInt:eq}
		\int_0^T \nabla F(u, X) \dd \Xb(u) := \lim_{\vert \partition\vert \to 0} \sum_{[s,t]\in\partition} \sum_{k=1}^{n} \langle\nabla^k F(s, X),\X^k_{s,t}\rangle,
	\end{align}
	is a well defined limit.
Moreover if $F$ satisfies Corollary \ref{cor:remainderestimate} to parameter $\tilde n$ such that $\alpha + (\tilde n-1)\alpha^2>1$, then 
\begin{align*}
F(T, X) = F(0, X) + \int_0^T DF(u, X) \dd u + \int_0^T \nabla F(u, X) \dd \Xb(u).
\end{align*}
\end{theo}

\begin{proof}
	We show existence of the rough integral by adapting the proof of \parencite[Proposition 4.10] {contPathwiseIntegrationChange2019} to our path-dependent  setting. 
Set for $(s,t)\in\Delta_T$, $k=1,\dots, n$,
\begin{align}\label{notationremainer}
	\Xi_{s,t}^X := \sum_{k=1}^n \langle \nabla^k F(s, X), \X_{s,t}^k \rangle.
\end{align}	
As usual in rough path theory \eqref{theo:roughfuncInt:eq} follows from the sewing lemma (compare e.g.\ \parencite[Theorem 4.3]{lyonsDifferentialEquationsDriven2007}) once we show that for every $(s,u), (u,t)\in\Delta_T$, 
	\begin{align}\label{theo:roughfuncInt:eq:1}
		\vert \Xi_{s,t}-\Xi_{s,u}-\Xi_{u,t}\vert \lesssim \vert t-s\vert^\theta 
	\end{align}
	for some $\theta >1$.
Recalling that $\nabla^k F(s, X)$ is symmetric, the reduced Chen relation \eqref{reducedChen} implies that
	\begin{align*}
	\la \nabla^k F(s,X), \X_{s,t}^k\ra 
	= \la \nabla^k F(s,X), \proj_k\sym{\Xb_{s,u}\otimes\Xb_{u,t}}\ra
	= \sum_{l=0}^k \la \nabla^k F(s,X), \X_{s,u}^{k-l}\otimes\X_{u,t}^{l} \ra.
	\end{align*}
Plugging that into $\Xi_{s,t}$ and interchanging the summation order, it follows that
\begin{align*}
\Xi_{s,t} - \Xi_{s,u} = \sum_{k=1}^n \sum_{l=k}^n \la\nabla^l F(s,X), \X_{s,u}^{l-k}\otimes\X_{u,t}^k \ra.
\end{align*}
Therefore 
\begin{align}\label{theo:roughfuncInt:eq:2}
 \Xi_{s,t}-\Xi_{s,u}-\Xi_{u,t} = - \sum_{k=1}^n \la R_{s,u}^{X, k}, \X_{u,t}^k\ra.
\end{align}
	Lemma \ref{lem:YXcontrolled} shows for $k=1, \dots, n$, that $R^{X,k}$ is $\alpha + (n-k)\alpha^2$-H\"older continuous.
	And recalling that $\X^k$ is $k\alpha$-H\"older continuous, we get that
	\begin{align*}
	\vert \la R_{s,u}^{X, k}, \X_{u,t}^k\ra \vert \lesssim \vert t-s\vert^{(k+1)\alpha + (n-k)\alpha^2}.
	\end{align*}
Since
	\begin{align}\label{theo:roughfuncInt:eq:1a}
		(k+1)\alpha + (n-k)\alpha^2 = n\alpha^2 + \alpha + k(\alpha - \alpha^2) > 2\alpha + (n-1)\alpha^2,
	\end{align}
and by assumption $2\alpha+(n-1)\alpha^2>1$, \eqref{theo:roughfuncInt:eq:1} now follows from \eqref{theo:roughfuncInt:eq:2}.

It is left to show the functional It\^o formula. Applying  Corollary \ref{cor:remainderestimate} to $F\in\mathbb{C}_b^{1,\tilde n}$ with $l=n$, shows that 
\begin{align*}
&\left\vert F(t, X) - F(s, X) - \int_s^t DF(u, X) \dd u - \sum_{k=1}^{n} \langle\nabla^k F(s, X),\X^k_{s,t}\rangle\right\vert \lesssim \vert t-s\vert^{1 + \alpha} + \vert t-s\vert^{(n+1)\alpha} + \vert t-s\vert^{\alpha+(\tilde n-1)\alpha^2}.
\end{align*}
By assumption $\alpha+(\tilde n-1)\alpha^2$ and $(n+1)\alpha$ are both greater one.
Together with the estimate from the sewing lemma it follows that $t\mapsto F(t, X) - F(0,X) - \int_0^t DF(u, X) \dd u - \int_0^t \nabla F(u,X)\dd\Xb(u)$ is $\tilde \theta$-H\"older continuous for some $\tilde \theta>1$. Consequently the map is constant zero.
\end{proof}

\begin{rem} Clearly $\tilde n\ge n+1$.
For Brownian sample paths the theorem can be applied with $n=2$ and $\tilde n = 4$. (Indeed $2\alpha + \alpha^2 <1 \Leftrightarrow \alpha > \sqrt{2} -1 \approx 0,41$ and $\alpha + 3\alpha^2 >1 \Leftrightarrow \alpha > (\sqrt{13} - 1)/6\approx 0,43$).
So for the existence of the integral it is sufficient that the functional $F$ has $3$-causal space derivatives, but for the It\^o formula we need $4$-causal space derivatives. This additional regularity is comparable to the regularity change in the standard setting \parencite[Lemma 4.1, Proposition 5.8]{RoughBook}. But there are regimes of $\alpha$ where the the change in regularity exceeds one. For example for $\alpha\in(\sqrt{2}-1, (\sqrt{16}-1)/6]$, it holds $n+1=3$ and $\tilde n=5$. This gap increases for $\alpha\to 0$.
It remains an open question if the loss of regularity from $\alpha$ to $\alpha^2$ in Corollary \ref{cor:remainderestimate} can be circumvented.
\end{rem}

\begin{exmp} 
We can easily change to  
\begin{align*}
\tilde\X^{n}_{s,t} := \X^{n}_{s,t} - \frac{1}{n!} \mu((s,t]),
\end{align*}
for a symmetric tensor-valued measure $\mu = \sum_{w} \mu^w e_w$ (over words $w$ of length $n$ in the alphabet $\A$), such that $\mu^w$ 
are finite signed measures with no atoms. Then it is immediate that
\begin{align*}
F(T, X) &= F(0, X) + \int_0^T DF(u, X) \dd u + \int_0^T \nabla F(u, X) \dd \Xb(u) - \frac{1}{n!} \int_0^T \la\nabla^{n}F(u, X), \dd\mu(u)\ra.
\end{align*}
The measure $\mu$ could be for example a suitable notion of finite $p$-variation, cf.\ \parencite[Definition 4.1]{contPathwiseIntegrationChange2019} or the stochastic quadratic variation if $X$ is a sample path of a semimartingale. 
\end{exmp}

\subsection*{Acknowledgement} The author would like to thank Christa Cuchiero, Xin Guo and Francesca Primavera for making available the slides of Christa Cuchiero's talk on \parencite{cuchieroFunctionalItoformulaTaylor} at TU Berlin, 2023. Moreover, Nicolas Perkowski for his helpful comments. 
\printbibliography
\end{document}